\documentclass[a4paper]{amsart}
\RequirePackage{amsmath, amsfonts, amssymb, amsthm}
\RequirePackage[utf8]{inputenc}
\RequirePackage{mathrsfs}
\RequirePackage[english]{babel}
\RequirePackage[T1]{fontenc}
\RequirePackage{mathpple}
\RequirePackage{tikz-cd}
\RequirePackage{mathtools}
\RequirePackage{stmaryrd}
\RequirePackage{xcolor}
\RequirePackage{faktor}
\usepackage{mathrsfs}  
\usepackage{stmaryrd}  
\usepackage{multirow}
\usepackage{array}
\usepackage{enumitem}

\RequirePackage[%
left = \glqq,%
right = \grqq,%
leftsub = \glq,%
rightsub = \grq%
]{dirtytalk}

\usepackage{setspace}
\newtheorem{thm}{Theorem}[section]
\newtheorem{prop}[thm]{Proposition}
\newtheorem{lem}[thm]{Lemma}
\newtheorem{cor}[thm]{Corollary}

\newtheorem*{claim}{Claim}
\newtheorem{problem}{Problem}
\theoremstyle{plain}
\theoremstyle{remark}
\newtheorem{rmk}[thm]{Remark}

\newtheorem{defin}[thm]{Definition}

\DeclareMathOperator{\lh}{lh}

\DeclareMathOperator{\supp}{supp}

\DeclareMathOperator{\Ch}{Ch}

\newcommand{\NNN}{\mathbb{N}}
\newcommand{\RRR}{\mathbb{R}}

\newcommand{\ZZZ}{\mathbb{Z}}

\newcommand{\Ss}{\mathcal{S}}

\date{\today}

\begin{document}

\title[On the $A$-invariance of the hereditary Baire property]{On the $A$-invariance of the hereditary Baire property and related results}
\author[M.\ Krupski]{Mikołaj Krupski}
\address{Institute of Mathematics\\ University of Warsaw\\ ul. Banacha 2\\
02--097 Warszawa, Poland }
\email{mkrupski@mimuw.edu.pl}
\author[K.\ Kucharski]{Kacper Kucharski}
\address{Institute of Mathematics\\ University of Warsaw\\ ul. Banacha 2\\
02--097 Warszawa, Poland }
\email{k.kucharski6@uw.edu.pl}

\thanks{The authors were partially supported by the NCN (National Science Centre, Poland) research Grant no. 2020/37/B/ST1/02613}

\begin{abstract}
We prove that if $X$ and $Y$ are first-countable perfect spaces such that the free Abelian topological groups $A(X)$ and $A(Y)$ are topologically isomorphic, then $X$ is a hereditarily Baire space if and only if $Y$ is hereditarily Baire as well. We also establish that for any Tychonoff space, if there exists a continuous linear surjection of the space $C_p(X)$ onto the space $C_p(Y)$ and the space $X$ is either strongly $\sigma$-scattered or has property $(\kappa)$, then $Y$ also satisfies those properties. Additionally, we obtain the following result: if $X$ and $Y$ are Tychonoff spaces in which every closed set has a $W$-point, and if the free Abelian topological groups $A(X)$ and $A(Y)$ are topologically isomorphic, then $X$ is scattered if and only if $Y$ is scattered.
\end{abstract}



\maketitle

\section{Introduction}
Unless stated otherwise, all spaces
under consideration are assumed to be Tychonoff.
For a space $X$, $C_p(X)$ denotes the space of continuous real-valued functions on $X$ endowed with the topology of pointwise convergence, while $F(X)$ (respectively, $A(X)$) denotes the free (respectively, free Abelian) topological group on $X$. For definitions and terminology, we refer the reader to Section 2. One of the main lines of research in the theory of $C_p$-spaces and the theory of free topological groups is the question of which topological properties of a space $X$ are encoded in the linear structure of $C_p(X)$, respectively, in the topological-group structure of $F(X)$ or $A(X)$ (see \cite{Ar}). Such properties are referred to as $l$-invariants, $M$-invariants, $A$-invariants, respectively. Finding them is important, as they serve as a tool for distinguishing function spaces or free (free Abelian) topological groups. It will be convenient to introduce the following terminology:
\begin{defin}
We say that spaces $X$ and $Y$ are \textit{$l$-equivalent} if the spaces $C_p(X)$ and $C_p(Y)$ are linearly homeomorphic. Similarly, we say that $X$ and $Y$ are \textit{$M$-equivalent} (respectively, \textit{$A$-equivalent}) if the groups $F(X)$ and $F(Y)$ (respectively, $A(X)$ and $A(Y)$) are topologically isomorphic. A topological property $P$ is said to be \textit{$M$-invariant}, \textit{$A$-invariant}, \textit{$l$-invariant}, respectively whenever it is invariant under the $M$-equivalence, $A$-equivalence, $l$-equivalence relation, respectively.
\end{defin}

It is well known that we have the following relationship between the equivalence relations defined above (see \cite{Ar}):
$$M\mbox{-equivalence}\Rightarrow A\mbox{-equivalence} \Rightarrow l\mbox{-equivalence}$$

In the present paper we are concerned with the following two open problems in the theory of function spaces and free topological groups.

\begin{problem}\label{problem1}
Is the hereditary Baire property $l$-invariant? Is it $M$-invariant or $A$-invariant? What happens in the realm of metrizable spaces?
\end{problem}

\begin{problem}\label{problem2}
Is scatteredness $l$-invariant? Is it $M$-invariant or $A$-invariant?
\end{problem}

Recall that a space $X$ is \textit{Baire} if the Baire category theorem holds in $X$, i.e., for any countable family $\{U_n:n\in \omega\}$ consisting of open dense subsets of $X$, the intersection $\bigcap_{n\in\omega} U_n$ is dense in $X$. A space $X$ is \textit{hereditarily Baire} if every closed subset of $X$ is Baire. We say that $X$ is \textit{scattered} if every nonempty subset of $X$ has an isolated point.
On the one hand, it was established by Baars, de Groot and Pelant \cite{BGP} that in the realm of metrizable spaces, complete metrizability is an $l$-invariant. On the other hand, it is known that, for metrizable spaces, being Baire is not invariant under the $l$-equivalence relation (see~\cite{Ma}). Although Problem~\ref{problem1} arises naturally in this context, to the best of our knowledge, it has never appeared in the literature. However, it has circulated among specialists in the field. Problem~\ref{problem2} is well known. For the $M$-equivalence relation, it was posed by Arhangel'skii (see~\cite{O}) in the 1980s. Perhaps the best result regarding this question was obtained by Baars~\cite{B}, who proved that scatteredness is an $l$-invariant for first-countable paracompact spaces.
At first sight, Problems~\ref{problem1} and \ref{problem2} seem unrelated. However, our next theorem (proved in Section~4) builds a bridge between them.

\begin{thm}\label{thm:Cp-scattered}
Suppose that there is a linear continuous surjection of $C_p(X)$ onto $C_p(Y)$. If $X$ is a countable union of closed scattered subspaces, then so is $Y$.
\end{thm}
A space that can be represented as a countable union of closed scattered subspaces is referred to as \textit{strongly $\sigma$-scattered}.
The above theorem generalizes a recent result of Eysen, Valov and Leiderman~\cite{ELV}, who proved an analogous statement under the additional assumption that both $X$ and $Y$ are metrizable. Using Theorem \ref{thm:Cp-scattered}, it is not difficult to show that an affirmative answer to Problem 1 implies an affirmative answer to Problem 2 (see Remark \ref{final_remark}).

In this paper we will prove the following:

\begin{thm}\label{hB-A-invariant}
 Let $X$ and $Y$ be first-countable perfect spaces. If $X$ and $Y$ are $A$-equivalent, then $X$ is hereditarily Baire if and only if $Y$ is hereditarily Baire.
\end{thm}
Recall that a space $X$ is \emph{perfect} if every closed subset of $X$ is a $G_\delta$-set. In particular, the hereditary Baire property is an $A$-invariant for metrizable spaces (cf. Problem \ref{problem1}). Concerning Problem \ref{problem2}, we will show the following:

\begin{thm}\label{thm_scattered_main}
 Let $X$ and $Y$ be first-countable spaces. If $X$ and $Y$ are $A$-equivalent, then $X$ is scattered if and only if $Y$ is scattered.
\end{thm}
In fact, we prove a slightly more general result in which the first-countability assumption is relaxed to the following condition: ``every closed subspace contains a $W$-point''. Theorem \ref{thm_scattered_main} is a partial refinement of the aforementioned result of Baars, showing that paracompactness is superfluous for the $A$-invariance of scatteredness.

We also answer an open question of Leiderman (personal communication) by proving that the so-called property $(\kappa)$ (for the definition, see Section 2) is preserved under linear continuous surjections of function spaces $C_p(X)$ (cf. Proposition \ref{proposition_property_kappa}).

\section{Preliminaries}
\noindent The first infinite ordinal number is denoted by $\omega$. The symbol  $\NNN$ stands for the set of all positive integers. If $X$ is a set and $n \in \omega$, then by $[X]^n$ we mean the set of all subsets of $X$ of cardinality $n$, and by $[X]^{\leq n}$, the set of all subsets of $X$ of cardinality at most $n$. Finally, by $[X]^{< \omega}$ we denote the set of all finite subsets of $X$.

\subsection{Free topological groups}

For a topological space $X$, \textit{the free (Abelian) topological group on $X$} is the (Abelian) topological group $F(X)$ (respectively, $A(X)$)
satisfying the following two conditions:
\begin{enumerate}[label=(\roman*)]
\item The space $X$ is a subspace of $F(X)$ (respectively, $A(X)$),
\item For any (Abelian) topological group $G$ and any continuous mapping 
$$f \colon X \to G$$ 
there exists a unique continuous homomorphism $\tilde{f} \colon F(X)\to G$ (respectively, $\tilde{f} \colon A(X)\to G$)
 such that $\tilde{f} \upharpoonright X = f$.
\end{enumerate}
Algebraically, $A(X)$ is the free Abelian group with $X$ being the set of generators. The elements of $A(X)$ are called \textit{words}. Every element of $A(X)$ distinct from the identity, is of the form
$\sum_{i=1}^n a_i x_i$, 
where $a_i\in \ZZZ\setminus\{0\}$, $x_i \in X$ and $n\in \NNN$. A word $\varphi = \sum_{i=1}^n a_i x_i$ is called {\it reduced} if all points $x_i \in X$ for $i = 1, \dots, n$ are distinct.
For any word $\varphi\in A(X)$ we define its \textit{length} $\lh(\varphi)$ by $\lh(\varphi) = \sum_{i = 1}^{n} |a_i|$, where $\varphi=\sum_{i=1}^n a_ix_i$ is the reduced form of $\varphi$.
The length of the identity in $A(X)$ (the empty word) is $0$ by definition.

For $n \in \omega$ we consider the sets:
$$
A_n(X) = \{\varphi \in A(X) \; \colon \; \lh(\varphi) \leq n\}.
$$
It is clear that if $n \leq m$, then $A_n(X) \subseteq A_m(X)$ and also that $A(X) = \bigcup_{n \in \omega} A_n(X)$. Moreover, straight from the definition, it follows that
\begin{equation}\label{eq:lh}
    \text{if } \varphi = \sum_{i = 1}^{n} a_i x_i \text{ for some } a_i \in \ZZZ \setminus \{0\} \text{ and distinct } x_i \in X, \text{ then } \lh(\varphi) \geq n.
\end{equation}

For more information on free (Abelian) topological groups we refer the reader to the book \cite[Chapter 7]{AT}.

\subsection{Topological games}

For a topological space $X$, {\it the strong Choquet game on $X$}, denoted by  $\Ch(X)$ (cf. \cite[Chapter 8.D]{K}), is a game with $\omega$-many innings, played alternately by Players I and II. Player I begins the game by picking a point $x_0 \in X$ and its open neighborhood $U_0$. Player II responds by picking an open set $V_0$ satisfying $x_0 \in V_0 \subseteq U_0$. In the second round, Player I chooses a pair $(x_1, U_1)$, where $x_1\in V_0$ and $U_1$ is an open neighborhood of $x_1$ contained in $V_0$. Player II responds with an open set $V_1$ such that $x_1\in V_1\subseteq U_1$. The game continues this way. The play
\[
\begin{array}{c||c c c c c c c c}
\multirow{2}{*}{\begin{tabular}{c} \text{I} \\ \text{II} \end{tabular}} 
& (x_0, U_0) & & (x_1, U_1) & & \cdots & (x_n, U_n) & & \cdots \\
\cline{1-9}
& & V_0 & & V_1 & \cdots & & V_n & \cdots \\
\end{array}
\]
is won by Player I if $\bigcap_{n \in \omega} V_n = \emptyset$; otherwise, Player II wins.\smallskip

By $\tau_X$, we denote the family of all nonempty open subsets of $X$. \emph{A strategy for Player I in the game $\Ch(X)$} is a map
$\sigma \colon \tau_X^{< \omega} \to X \times \tau_{X}$
defined inductively as follows:
\begin{itemize}
    \item $\sigma(\emptyset) = (x_0, U_0) \in X \times \tau_{X}$ such that $x_0 \in U_0$;
    \item Assuming $\sigma$ has been defined for sequences of length less than or equal to $n$, an $(n + 1)$-tuple $(V_0, \dots, V_{n}) \in \tau_X^{n + 1}$ is called {\it admissible} if the initial segment of the play
\[
\begin{array}{c||c c c c c c c}
\multirow{2}{*}{\begin{tabular}{c} \text{I} \\ \text{II} \end{tabular}} 
& (x_0, U_0) & & (x_1, U_1) & & \cdots & (x_{n-1}, U_{n-1}) \\
\cline{1-8}
& & V_0 & & V_1 & \cdots & & V_{n}\\
\end{array}
\]
is played according to the rules of the game $\Ch(X)$ and
$\sigma(V_0, \dots,V_i) = (x_i, U_i)$ 
for $i = 0, \dots, n - 1$. Given an admissible $(n +  1)$-tuple $(V_0, \dots, V_{n})$, we choose a pair $(x_{n}, U_{n}) = \sigma(V_0, \dots, V_{n})$ such that $x_{n} \in V_{n} \subseteq U_{n}$.
\end{itemize}
A strategy $\sigma$ is called {\it winning}, if Player I wins every run of the game $\Ch(X)$, provided she plays according to $\sigma$.\smallskip

The notion of the strong Choquet game provides a convenient tool to characterize first-countable perfect spaces that are hereditarily Baire, as was shown in \cite{D} (cf. \cite{T}). Namely, we have the following:
\begin{thm}[Debs {\cite[Proposition 1.2 and Proposition 2.7]{D}}]
Let $X$ be a space.
\begin{enumerate}
    \item If Player I does not have a winning strategy in the strong Choquet game $\Ch(X)$, then every $G_{\delta}$ subspace of $X$ is Baire.
    \item Assume that $X$ is first-countable and regular. If Player I has a winning strategy in the strong Choquet game $\Ch(X)$, then $X$ has a closed subspace homeomorphic to the rationals.
\end{enumerate}
\end{thm}

The hereditary Baire property for first-countable regular spaces also has a useful internal characterization due to van Douwen \cite{vD}, which generalizes the classical result of Hurewicz \cite{H}:
\begin{thm}[van Douwen \cite{vD}]
Let $X$ be a first-countable regular space. Then $X$ is hereditarily Baire if and only if $X$ has no closed subspace homeomorphic to the rationals.
\end{thm}

From the results above, it is easy to derive the following:

\begin{thm}\label{thm:choquet}
For a first-countable regular perfect space $X$, the following conditions are equivalent:
\begin{enumerate}
    \item[$(1)$] $X$ is not hereditarily Baire,
    \item[$(2)$] Player I has a winning strategy in the strong Choquet game $\Ch(X)$,
    \item[$(3)$] $X$ has a closed countable crowded subspace.
\end{enumerate}
\end{thm}

If we were to drop the assumption about the space $X$ being perfect, then the conditions above have the following relationship:
$$
(2) \Rightarrow (1) \Leftrightarrow (3).
$$
Moreover, Zsilinszky (cf. \cite[Example 3.4]{Z}) provided an example of a first-countable space $X$ with a closed copy of the rationals such that Player II has a winning strategy in the game $\Ch(X)$. This shows that to obtain a theorem such as the one above, one needs to assume something more beyond just first-countability.\medskip

In the sequel, we shall also consider the following two-player infinite game on a topological space $X$ at a fixed point $x\in X$, called the $W$-game and denoted by $W(X,x)$.
A play in $W(X,x)$ proceeds as follows: in the $n$-th round, Player I picks an open neighborhood $U_n$ of $x$, and Player II responds by choosing a point $x_n \in U_n$. Player I wins the play if the sequence $(x_n)_{n \in \omega}$ converges to $x$. Otherwise, Player II wins. The notion of a winning strategy for Player I in the game $W(X,x)$ is defined analogously to that in the game $\Ch(X)$.

We say that a point $x\in X$ is a \emph{$W$-point} if Player I has a winning strategy in the game $W(X,x)$, and a space $X$ is a \emph{$W$-space} if all of its points are $W$-points. The game $W(X,x)$ and the concept of a $W$-space were introduced by Gruenhage in \cite{G} as a natural generalization of first-countability.

\subsection{The support function}

Let $X$ and $Y$ be topological spaces. A multivalued mapping $S \colon X \to [Y]^{<\omega}$ is called {\it lower semi-continuous} (lsc for short) if the set
$$
S^{-1}(U) = \{x \in X \; \colon \; S(x) \cap U \neq \emptyset\}
$$
is open, whenever $U$ is an open subset of $Y$.

\begin{defin}\label{def:supplike}
Let $X$ and $Y$ be topological spaces. A multivalued mapping $S \colon X \to [Y]^{<\omega}$ is called {\it supp-like} if it satisfies the following conditions:
\begin{enumerate}
    \item[$(a)$] $S$ is lower semi-continuous.
    \item[$(b)$] The set $S(z)$ is nonempty for all $z \in X$.
\end{enumerate}
Two supp-like mappings $S \colon X \to [Y]^{<\omega}$ and $T \colon Y \to [X]^{< \omega}$ are called {\it linked} if they satisfy the following additional condition:
\begin{enumerate}
    \item[$(c)$] For all $x \in X$ and $y \in Y$ we have 
    $x \in T(S(x))$ and $y \in S(T(y))$.
\end{enumerate}
\end{defin}

Let $S \colon X \to [Y]^{<\omega}$ and $T \colon Y \to [X]^{< \omega}$ be linked supp-like mappings. In the sequel we shall consider the following two sets:
$$
\theta(y) = \{x \in T(y) \colon y \in S(x)\}\;\; \text{and} \;\; \theta(B) = \bigcup\{\theta(y) \colon y \in B\},
$$ 
for $y \in Y$ and $B \subseteq Y$. Points $x \in \theta(y)$ are called {\it $y$-related}, whereas points $x \in \theta(B)$ are called $B$-related. Note that by $(c)$, the set $\theta(y)$ is nonempty for all $y \in Y$.
\smallskip

Let $\phi \colon A(X) \to A(Y)$ be a continuous group monomorphism. For any $x \in X$, there exist distinct points $y_1, \dots , y_n \in Y$ and integers
$a_1, \dots, a_n \in \ZZZ \setminus \{0\}$ such that $\phi(x) = \sum_{i=1}^{n} a_i y_i$. Therefore, to each $x \in X$ we can assign
the following finite subset of $Y$: $$\supp_{\phi}(x)=\{y_1, \dots, y_n\}$$ called {\it the support of $x$}.

\begin{lem}\label{lem:supp_lsc}
Let $\phi \colon A(X)\to A(Y)$ be a continuous group monomorphism. Then the function $\supp_{\phi} \colon X \to [Y]^{<\omega}$ is lower semi-continuous.
\end{lem}

\begin{proof}
Let $V$ be an open subset of $Y$. We need to show that the set
$$
W = \supp_{\phi}^{-1}(V)=\{x\in X \colon \supp_{\phi}(x)\cap V \neq \emptyset\}
$$
is open. To this end, take any $x_0 \in W$. Fix $y_0 \in \supp_{\phi}(x_0)\cap V$. Shrinking $V$ if necessary, without loss of generality, we may assume that $\supp_{\phi}(x_0)\cap V=\{y_0\}$.
Let $g \colon Y\to \RRR$ be a continuous function satisfying $g\upharpoonright(Y \setminus V) = 0$ and $g(y_0) = 1$. By the definition of $A(Y)$, there is the unique extension of $g$ to a continuous homomorphism $\tilde{g} \colon A(Y) \to \RRR$. Consider the set
$U = \{x\in X: \tilde{g}(\phi(x)) \neq 0\}$.
The set $U$ is open in $X$, being the preimage of an open set in $\RRR$ under the continuous mapping 
$\tilde{g}\circ(\phi\upharpoonright X): X\to\RRR$ and, since for some $a_y \in \ZZZ \setminus\{0\}$,
$$
\tilde{g}\circ(\phi\upharpoonright X)(x_0)=\tilde{g}\left( \sum_{y\in\supp_{\phi}(x)}a_yy \right)= \sum_{y\in\supp_{\phi}(x)}a_y g(y) = a_{y_0} \neq 0,
$$
we have $x_0 \in U$. We claim that $U \subseteq W$. Indeed, if $\supp_{\phi}(x) \cap V = \emptyset$ then by the definition of $g$, we have $\tilde{g}\circ(\phi \upharpoonright X)(x)=0$.
\end{proof}

\begin{prop}\label{prop:suppA}
If $\phi \colon A(X) \to A(Y)$ is a topological isomorphism, then the functions $\supp_{\phi}$ and $\supp_{\phi^{-1}}$ are supp-like and linked.
\end{prop}

\begin{proof}
Condition $(a)$ for both functions $\supp_{\phi}$ and $\supp_{\phi^{-1}}$ holds by Lemma \ref{lem:supp_lsc}, whereas condition $(b)$ follows immediately from the definition of $\supp_{\phi}$ or $\supp_{\phi^{-1}}$ and the fact that $\phi$ is a group isomorphism. This proves that both mappings are supp-like.\smallskip

To prove that they are linked, i.e., that the condition $(c)$ holds, assume that there exists $x \in X$ which does not belong to the set $F = \supp_{\phi^{-1}}(\supp_{\phi}(x))$. Denote
$\phi(x) = \sum_{i = 1}^{n} a_i y_i$,
and let
$\phi^{-1}(y_i) = \sum_{j = 1}^{n_i} b_j^i x_j^i$ for $i = 1, \dots, n$.
\smallskip

Our assumption states that $x \neq x_j^i$ for $i = 1, \dots, n$ and $j = 1, \dots, n_i$. Since the set $F$ is finite, there exists a continuous function $g \colon X \to \RRR$ satisfying $g(x) = 1$ and $g(F) \subseteq \{0\}$. Since the real line $\RRR$ is an Abelian topological group, there exists a unique continuous homomorphism $\tilde{g} \colon A(X) \to \RRR$ extending $g$. Note that because $\phi$ is an isomorphism, we have $\tilde{g} = \tilde{g} \circ \phi^{-1} \circ \phi$. In particular:
\begin{equation}
    \begin{split}
        1 & = \tilde{g}(x) = (\tilde{g} \circ \phi^{-1} \circ \phi)(x) = 
        (\tilde{g} \circ \phi^{-1})\left(\sum_{i = 1}^{n} a_i y_i\right) =\\
        & = \tilde{g}\left(\sum_{i = 1}^{n} a_i \phi^{-1}(y_i)\right) =
        \tilde{g}\left(\sum_{i = 1}^{n} \sum_{j = 1}^{n_i} a_i b_j^i x_j^i\right) = \sum_{i = 1}^{n} \sum_{j = 1}^{n_i} a_i b_j^i g(x_j^i) = 0,
    \end{split}
\end{equation}
which yields the desired contradiction. An analogous argument shows that $y \in \supp_{\phi}(\supp_{\phi^{-1}}(y))$ for all $y \in Y$.
\end{proof}

The following theorem (see e.g. \cite[Corollary 7.4.4]{AT}) is essential for our reasoning:

\begin{thm}\label{thm:compact}
Let $K$ be a compact subspace of $A(X)$. Then there exists $n \in \omega$ such that $K \subseteq A_n(X)$.
\end{thm}

In fact, we will need the following corollary of the above:

\begin{cor}\label{cor:bdd_supp}
Assume that for spaces $X$ and $Y$ there exists a topological isomorphism of groups $\phi \colon A(X) \to A(Y)$.
Then for every $W$-point $x \in X$ there exists an open neighborhood $V_x$ of $x$ and a natural number $M$ such that $|\supp_{\phi}(z)| \leq M$ for all $z \in V_x$.
\end{cor}

\begin{proof}
Let $x \in X$ be a $W$-point and let $\sigma$ be a winning strategy for Player I in the $W$-game $W(X,x)$.
Striving for a contradiction, assume that whenever $V$ is an open neighborhood of $x$ and $M$ is any natural number, there is $z \in V$ with $|\supp_{\phi}(z)| > M$. Using the strategy $\sigma$ define by induction a sequence $(x_n)_{n \in \NNN}$ convergent to $x$ such that $|\supp_{\phi}(x_n)| > n$ for all $n$.\smallskip

Assume that the points $x_1, \dots, x_{n - 1}$ satisfying $|\supp_{\phi}(x_i)| > i$ for $i < n$ are already defined. Let $U_{n} = \sigma(x_1, \dots, x_{n - 1})$ be an open neighborhood of $x$. Pick $x_{n} \in U_{n}$ with $|\supp_{\phi}(x_{n})| > n$, which is possible by our assumption on $x$. Since the strategy $\sigma$ is winning, the sequence $(x_n)_{n = 1}^{\infty}$ is convergent to $x$, which finishes the construction.\smallskip

Now, the set $K = \{x_n \colon n \in \NNN\} \cup \{x\}$ is compact, and so the set $\phi(K)$ is compact as well. It follows from Theorem \ref{thm:compact} that there exists $N \in \NNN$ such that $\phi(K) \subseteq A_N(Y)$. In particular, for all $z \in K$, if $\phi(z) = \sum_{i = 1}^{M} a_i y_i$ is a reduced word, then, by (\ref{eq:lh}) we have $M \leq N$. This is a contradiction, because if $n > N$, then by our choice of $x_n$ we have $|\supp_{\phi}(x_n)| > n > N$ and hence the number of summands in a reduced word $\phi(x_n)$ is strictly greater than $N$.
\end{proof}


\section{The hereditary Baire property and the free Abelian topological groups}\label{sec:abelian}

\noindent
In this section we will prove Theorem~\ref{hB-A-invariant}.
Before proceeding to the proof, we need some preparation. A space is {\it crowded} if it has no isolated points.

\begin{lem}\label{lem:case1}
Let $X$ and $Y$ be topological spaces. Assume that the space $Y$ contains a countable crowded $G_{\delta}$ subspace $Q$. Assume further that there exist supp-like linked mappings $S \colon X \to [Y]^{< \omega}$ and $T \colon Y \to [X]^{< \omega}$. Moreover, assume that the following condition holds:
\begin{equation}\tag{$\star$}
\begin{array}{c}
\text{There is a nonempty relatively open set } W \subseteq Q \text{ such that for every } y \in W \\
\text{and every relatively open neighborhood } U \subseteq W \text{ of } y,\ \text{if } x \in \theta(y), \\
\text{then } x \text{ is not isolated in the set } \theta(U).
\end{array}
\end{equation}

Then Player I has a winning strategy in the strong Choquet game $\Ch(X)$.
\end{lem}

\begin{proof}
We shall construct a winning strategy $\sigma$ for Player I in the game $\Ch(X)$ by induction.
Let $\{E_n \colon n \in \omega\}$ be a family of open subsets of $X$ such that $Q = \bigcap_{n \in \omega} E_n$ and let $\{q_n \colon n \in \omega\}$ be an enumeration of $Q$.
Since $Q$ is crowded, we can pick a point $y_0 \in W$ different from $q_0$. Choose an open neighborhood $B_0 \subseteq E_0$ of the point $y_0$ satisfying:
\begin{equation}
q_0 \notin B_0 \;\; \text{and} \;\; B_0 \cap Q \subseteq W.
\end{equation}
For convenience, denote $U_{-1} = B_0$.
Pick a $y_0$-related point $x_0 \in \theta(y_0)$ and let $D_0 = S^{-1}(B_0)$. Then $x_0 \in D_0$, since $y_0 \in B_0$. Define $\sigma(\emptyset) = (x_0, D_0)$. Now, assume that $V_0$ is an open set with $x_0 \in V_0 \subseteq D_0$. Define the set $U_0 = B_0 \cap T^{-1}(V_0)$ and note that it is nonempty, since $y_0$ lies both in $B_0$ and $T^{-1}(V_0)$. Finally, notice that there exists $y_1 \in U_0 \cap Q$ outside the set $S(x_0) \cup \{q_1\}$ such that $\theta(y_1) \cap V_0 \neq \emptyset$.\smallskip

Indeed, by $(\star)$ the point $x_0$ is not isolated in the set $\theta(U_0 \cap Q)$. From this and the fact that $T(y)$ is finite for all $y \in Y$, we infer that $\theta(y) \cap V_0\neq \emptyset$ for infinitely many points $y \in U_0 \cap Q$. Now, the set $S(x_0) \cup \{q_1\}$ is finite, so we may pick $y_1 \in (U_0 \cap Q) \setminus (S(x_0) \cup \{q_1\})$, as desired.\smallskip

Assume that for some $n \geq 1$, the strategy $\sigma$ is already constructed for all admissible $n$-tuples $(V_0, \dots, V_{n - 1}) \in \tau_X^n$, i.e., there is the beginning of the play:
\begin{equation}
\begin{array}{c||c c c c c c c c}
\multirow{2}{*}{\begin{tabular}{c} \text{I} \\ \text{II} \end{tabular}} 
& (x_0, D_0) & & (x_1, D_1) & & \cdots & (x_{n-1}, D_{n-1}) & &  (x_{n}, D_{n}) \\
\cline{1-9}
& & V_0 & & V_1 & \cdots & & V_{n-1}\\
\end{array}
\end{equation}
in the Choquet game on $X$, where Player I plays according to $\sigma$, i.e.,
$$
\sigma(V_0, \dots, V_i) = (x_{i + 1}, D_{i + 1}) \; \text{for} \; i = 0, \dots, n - 1.
$$
Moreover, assume that there exist
\begin{itemize}
    \item points $y_0, \dots, y_{n} \in W$,
    \item open neighborhoods $B_0, \dots, B_{n}$ of $y_0, \dots, y_{n}$ respectively,
    \item open sets $U_0, \dots, U_{n - 1} \subseteq Y$,
\end{itemize}
such that the following conditions hold:
\begin{enumerate}
    \item[(a)] $x_i$ is a $y_i$-related point,
    \item[(b)] $y_i \in B_i \subseteq U_{i - 1} \cap E_i$,
    \item[(c)] $D_i = S^{-1}(B_i)$,
    \item[(d)] $q_i \notin B_i$,
\end{enumerate}
for $i = 0, \dots, n$ and
\begin{enumerate}
    \item[(e)] $U_{i} = B_i \cap T^{-1}(V_{i})$,
\end{enumerate}
for $i = 0, \dots, n - 1$.

Now fix any admissible tuple $(V_0, \dots, V_{n}) \in \tau_X^{n + 1}$. Below we will describe how to construct $\sigma\big((V_0, \dots, V_{n})\big)$ alongside the point $y_{n + 1}$ and its neighborhoods $U_{n}$ and $B_{n + 1}$. Let
\begin{equation}\label{eq:Un}
    U_n = B_n \cap T^{-1}(V_n)
\end{equation}
and notice that by a similar argument as in the base step, we have the following:
%
%
\begin{equation}\label{eq:ind}
\begin{split}
\text{The}&\text{re exists a point } y_{n+1} \in U_n \cap Q \text{ outside the set }\\
& S(x_n) \cup \{q_{n+1}\},
\text{ such that }
\theta(y_{n+1}) \cap V_n \neq \emptyset .
\end{split}
\end{equation}
Pick an open neighborhood $B_{n + 1}$ of the point $y_{n + 1}$ with 
\begin{equation}
    q_{n + 1} \notin B_{n + 1} \text{,} \;\; B_{n + 1} \subseteq U_n \;\; \text{and} \;\; B_{n + 1} \subseteq E_{n + 1}
\end{equation}
and fix $x_{n + 1} \in \theta(y_{n + 1}) \cap V_n$. This is possible by (\ref{eq:ind}). Let
\begin{equation}
    D_{n + 1} = S^{-1}(B_{n + 1}).
\end{equation}
The set $D_{n + 1}$ is open and contains $x_{n + 1}$. Finally, put $\sigma\big((V_0, \dots, V_n)\big) = (x_{n + 1}, D_{n + 1})$. It is straightforward to check that the conditions (a) -- (d) hold for $i = n + 1$. The condition (e) holds for $i = n$, by (\ref{eq:Un}). This finishes the inductive construction of the strategy $\sigma$.
\smallskip

We claim that the strategy $\sigma$ defined as above is winning. Striving for a contradiction, suppose it is not, and fix a play
\[
\begin{array}{c||c c c c c c c c}
\multirow{2}{*}{\begin{tabular}{c} \text{I} \\ \text{II} \end{tabular}} 
& (x_0, D_0) & & (x_1, D_1) & & \cdots & (x_n, D_n) & & \cdots \\
\cline{1-9}
& & V_0 & & V_1 & \cdots & & V_n & \cdots \\
\end{array}
\]
where Player I plays according to the strategy $\sigma$ and loses, i.e., $\bigcap_{n \in \omega} V_n \neq \emptyset$.
Pick any point $x_{\infty} \in \bigcap_{n \in \omega} V_n$.
Since, by condition (c), each set $V_n$ is contained in $S^{-1}(B_n)$, we have $S(x_{\infty}) \cap B_n \neq \emptyset$ for all $n \in \omega$. The set $S(x_{\infty})$ is finite, so there is $y_{\infty} \in S(x_{\infty})$ lying in all sets $B_n$.\smallskip

Now, notice that by (b), we have $y_{\infty} \in \bigcap_{n \in \omega} B_n \subseteq \bigcap_{n \in \omega} E_n = Q$.
However, this cannot be the case, by virtue of condition (d), in which we promised that $y_{\infty} \neq q_n$ for all $n$. This finishes the proof.
\end{proof}
 
\begin{lem}\label{lem:main_tool}
Let $X$ and $Y$ be first-countable spaces such that $Y$ has a countable crowded $G_{\delta}$-subspace $Q$. Assume that there exist supp-like linked mappings $S \colon X \to [Y]^{< \omega}$ and $T \colon Y \to [X]^{< \omega}$ with the following additional property:
\begin{equation}\tag{$\star \star$}
    \begin{aligned}
        \text{There}& \; \text{exists a natural number} \;
        M \in \NNN 
        \; \text{and a relatively open} \;\\ & \text{subset} \;
        V \subseteq Q
        \; \text{such that} \;
        T(y) \in [X]^{M}
        \; \text{for all} \;
        y \in V
    \end{aligned}
\end{equation}
Then the condition $(\star)$ from Lemma \ref{lem:case1} holds. In particular, Player I has a winning strategy in the strong Choquet game $\Ch(X)$.
\end{lem}

\begin{proof}
Striving for a contradiction, assume that condition $(\star)$ does not hold. Then we can find
$y_1 \in V$ and its open neighborhood $U_1 \subseteq V$ such that there exists $x \in \theta(y_1)$ which is isolated in $\theta(U_1)$. Denote $T(y_1) = \{x_1^1, \dots, x_M^1\}$ and without loss of generality assume that $x_1^1$ is the point $x$ as above.\smallskip

Let $D^1_1, \dots, D^1_M$ be pairwise disjoint open neighborhoods of $x^1_1, \dots, x^1_M$, respectively, such that $D_1^1$ isolates
$x_1^1$ in $\theta(U_1)$, that is,
$$D_1^1 \cap \theta(U_1) = \{x_1^1\}.$$
Finally, let $W_1$ be any nonempty relatively open subset of $V$ satisfying
$$
W_1 \subseteq \bigcap_{i = 1}^M T^{-1}(D^1_i) \cap U_1 \quad \text{and} \quad W_1 \cap S(x_1^1) = \emptyset.
$$
Such a set exists, because $Q$ is crowded and the set $S(x_1^1)$ is finite.
The set $W_1$ is a nonempty relatively open subset of $Q$ with the following property:
\begin{equation}\label{eq9}
    \text{if} \;\; y \in W_1, \; \text{then} \;\; \theta(y) \cap D_1^1 = \emptyset.
\end{equation}
In other words, all $y$-related points miss the set $D_1^1$ for every $y \in W_1$. To show \eqref{eq9}, fix $y \in W_1$. Then, for each $1\leq i\leq M$, we have
\begin{equation}\label{eq10}
|T(y) \cap D_i^1| = 1,
\end{equation}
because by $(\star\star)$, $T(y)\in [X]^M$ and, since $y\in W_1$, the set $T(y)$ meets every set in the pairwise disjoint family $\{ D^1_i:i=1,\ldots , M\}$. Aiming at a contradiction, suppose that
$\theta(y) \cap D_1^1 \neq \emptyset$. Then we infer from \eqref{eq10} and $D_1^1\cap \theta(U_1)=\{x^1_1\}$ that $x^1_1\in \theta(y)$. But then, according to the definition of $\theta(y)$, we get $y\in S(x^1_1)$, violating the choice of $y\in W_1$. Condition \eqref{eq9} is proved.

Since
$$
\emptyset \neq \theta(y) \subseteq T(y) \subseteq \bigcup_{i = 1}^M D_1^i,
$$
we infer from \eqref{eq9} that $\theta(y)\subseteq D_2^1\cup \dotsb \cup D_M^1$ for all $y \in W_1$.
\smallskip

Using the failure of $(\star)$ again, we can find $y_2\in W_1$ and an open neighborhood $U_2\subseteq W_1$ of $y_2$ such that there exists $x\in \theta(y_2)$ which is isolated in $\theta(U_2)$. Set $T(y_2)=\{x^2_1,\ldots , x^2_M\}$. Note that, since $y_2\in W_2\subseteq W_1$, we have $|T(y_2)\cap D^1_i|=1$ for each $1\leq i\leq M$ (cf. \eqref{eq10}). Thus, without loss of generality we can assume that $T(y_2)\cap D^1_i=\{x^2_i\}$, for $i=1,\ldots , M$. Since $x\in \theta(y_2)$, according to \eqref{eq9}, $x\neq x^2_1$. Without loss of generality, we can assume that $x=x^2_2$, i.e., the point $x^2_2$ is isolated in $\theta(U_2)$.
Let $D^2_2$ be an open subset of $D^1_2$ such that
$$D_2^2 \cap \theta(U_2) = \{x_2^2\}.$$
For $i\in \{1,\ldots ,M\}\setminus\{2\},\; \text{ set }\;
D_i^2=D_i^1.$

Arguing as above, we can find a nonempty open subset $W_2$ of $U_2$ such that if $y\in W_2$, then $\theta(y)\cap D^2_2=\emptyset$ (cf. \eqref{eq9}). This together with property \eqref{eq9} implies that $\theta(y)\subseteq D_3^2\cup\dotsb \cup D_M^2$ for all $y\in W_2$.

\smallskip

Repeating an analogous argument recursively finitely many times, we construct, for each $j=1,2,\ldots ,M$, nonempty open sets $D_1^j,\ldots D_M^j$ such that $D_i^1\supseteq \dotsb \supseteq D_i^M$ for $i=1,\ldots ,M$ and open sets $W_1\supseteq W_2\supseteq \dotsb \supseteq W_M$ such that
\begin{equation}\label{eq11a}
 T(y)\subseteq D_1^i\cup\dotsb \cup D_M^i\; \text{ for }\;y\in W_i
\end{equation}
and,
for each $i=1, \ldots ,M$,
\begin{equation}\label{eq11b}
 \text{if }\; y\in W_i \;\text{ then } \;\theta(y)\cap (D_1^i\cup\dotsb \cup D_i^i)=\emptyset.
\end{equation}

However, for $i=M$, conditions \eqref{eq11a} and \eqref{eq11b} mean that $\theta(y)=\emptyset$ for $y\in W_M$. This contradicts the assumption that $T$ and $S$ are linked (cf. the remark following Definition \ref{def:supplike}).
\end{proof}

Now we can prove Theorem \ref{hB-A-invariant}.

\begin{proof}[Proof of Theorem \ref{hB-A-invariant}]
Let $\phi \colon A(X) \to A(Y)$ be a topological isomorphism of groups. By symmetry, it is enough to prove that if $Y$ is not a hereditarily Baire space, then $X$ is not hereditarily Baire either. Let us assume that $Y$ is not hereditarily Baire. By Theorem \ref{thm:choquet}, the space $Y$ has a closed copy of the rationals; denote it by $Q$. We will prove that Player I has a winning strategy in the strong Choquet game, which by Theorem \ref{thm:choquet}, will finish the proof.\smallskip

By Proposition \ref{prop:suppA}, the multivalued mappings $\supp_{\phi}$ and $\supp_{\phi^{-1}}$ are supp-like and linked. By Corollary \ref{cor:bdd_supp}, there exist a natural number $M \in \NNN$ and an open subset $W$ of $Q$ such that $\supp_{\phi^{-1}}(y)\in[X]^{\leq M}$ for all $y\in W$. Define
$$
N = \max\{|\supp_{\phi^{-1}}(y)| \; \colon \; y \in W\}
$$
and fix $y \in W$ such that $|\supp_{\phi^{-1}}(y)| = N$.
Denote $\supp_{\phi^{-1}}(y) = \{x_1, \dots, x_N\}$ and fix a family $\{B_1, \dots, B_N\}$ of pairwise disjoint open subsets of $X$ such that $x_i \in B_i$ for $i = 1, \dots, N$. Finally, define an open set:
$$
V = W \cap \bigcap_{i = 1}^{N} \supp_{\phi^{-1}}^{-1}(B_i).
$$
The set $V$ is an open, nonempty (because $y\in V$) subset of $Q$ with the property that for all $z \in V$ we have $|\supp_{\phi^{-1}}(z)| = N$, i.e., the set $V$ along with the supp-like linked mappings $\supp_{\phi}$ and $\supp_{\phi^{-1}}$ satisfy the assumption of Lemma \ref{lem:main_tool}, which yields the desired winning strategy for Player I in the strong Choquet game $\Ch(X)$. This finishes the proof of the theorem.
\end{proof}


\section{Scattered-like properties, continuous linear surjections of $C_p$-spaces, and further remarks on $A$-equivalence}\label{sec:scattered}
\noindent The main tool used in the proofs below is the support map for linear mappings between function spaces. Suppose that for $X$ and $Y$ there exists a continuous linear mapping $\phi \colon C_p(X) \to C_p(Y)$. For every $y \in Y$, we define $\supp_{\phi}(y)$ as the set \footnote{To denote the support of a linear mapping between function spaces, we use the same symbol as for the support in the context of free topological groups (cf. Section 2.3). This should not cause any confusion, since we will never consider function spaces and free topological groups simultaneously.} of all points $x \in X$ such that for every open neighborhood $U$ of $x$ there exists a function $f \in C_p(X)$ satisfying $f(X \setminus U) \subseteq \{0\}$ and $\phi(f)(y) \neq 0$.
The key properties of the multivalued mapping $y\mapsto \supp_{\phi}(y)$ are given in the following well-known lemma (see \cite[Lemma 6.8.2]{vM}):

\begin{lem}
Let $\phi \colon C_p(X) \to C_p(Y)$ be a continuous linear surjection. Then for every $y\in Y$ the set $\supp_\phi(y)$ is a nonempty finite subset of $X$ and the multivalued function $\supp_{\phi} \colon Y \to [X]^{< \omega}$ is supp-like.
\end{lem}

The next lemma belongs to the mathematical folklore; we include a proof for completeness.

\begin{lem}\label{lem:fin}
Suppose that $\phi \colon C_p(X) \to C_p(Y)$ is a continuous linear surjection. If $F$ is a finite subset of $X$, then the cardinality of the set
$\{y \in Y  \colon  \supp_{\phi}(y) \subseteq F\}$ does not exceed the size of $F$.
\end{lem}

\begin{proof}
Denote $F = \{x_1, \dots, x_n\}$. Aiming at a contradiction, suppose that there are at least $n + 1$ distinct points $y_1, \dots, y_{n + 1}$ satisfying $\supp_{\phi}(y_i) \subseteq F$.

Define a mapping $T \colon \RRR^n \to \RRR^{n + 1}$ as follows: for $(t_1, \dots, t_n) \in \RRR^n$, choose any function $f \in C_p(X)$ such that $f(x_i) = t_i$ for $i = 1, \dots , n$ and let
\[
T(t_1, \dots, t_n) = (\phi(f)(y_1), \dots, \phi(f)(y_{n + 1})).
\]
It is a well-known fact that if two functions $f$ and $g$ agree on $\supp_{\phi}(y)$ for some point $y$, then the functions $\phi(f)$ and $\phi(g)$ have the same values at $y$ (see \cite[Lemma 6.8.1]{vM}). In particular, it follows that the value of $T$ at $(t_1, \dots, t_n)$ does not depend on the choice of $f$. So the mapping $T$ is well defined.\smallskip

It is easy to verify that $T$ is linear. Let us check that it is also surjective. To this end, pick any point $(s_1, \dots, s_{n + 1}) \in \RRR^{n + 1}$ and $g \in C_p(Y)$ such that $g(y_i) = s_i$ for $i = 1, \dots, n + 1$. Since $\phi$ is surjective, there exists $f \in C_p(X)$ such that $\phi(f) = g$. Clearly,
\[
T(f(x_1), \dots, f(x_n)) = (s_1, \dots, s_{n + 1}).
\]
This yields a contradiction, since $\mathbb{R}^n$ cannot be mapped linearly onto $\mathbb{R}^{n+1}$.
\end{proof}

It turns out that the same assertion as in Lemma \ref{lem:fin} holds for the support of topological isomorphisms between free (Abelian) topological groups.

\begin{lem}\label{lem:fin_grp}
Suppose that $\phi \colon A(X) \to A(Y)$ is a topological isomorphism. If $F$ is a finite subset of $Y$ then the cardinality of the set
$\{x \in X  \colon  \supp_{\phi}(x) \subseteq F\}$ does not exceed the size of $F$.
\end{lem}

\begin{proof}
The proof is essentially the same as that of Lemma \ref{lem:fin}, so we only point out the main differences.\smallskip

If $f \in C_p(Y)$, then by property (ii) in the definition of $A(Y)$ and the fact that $\RRR$ is an Abelian topological group, there exists a unique continuous group homomorphism $\Tilde{f} \colon A(Y) \to \RRR$ extending $f$. Then the function $\Tilde{f} \circ \phi \colon A(X) \to \RRR$ is also a continuous group homomorphism, which, moreover, is the unique continuous group homomorphism extending the continuous function $\varphi(f) = (\Tilde{f} \circ \phi)\upharpoonright X$. To repeat the proof of the previous lemma, one needs to check that the assignment
$$
C_p(Y) \ni f \mapsto \varphi(f) \in C_p(X)
$$
is a linear surjection and that the following claim holds:
\begin{claim}
If $f, g \in C_p(Y)$ agree on the finite set $\supp_{\phi}(x)$ for some $x \in X$, then $\varphi(f)(x) = \varphi(g)(x)$.
\end{claim}
To verify this claim, let $\phi(x) = \sum_{i = 1}^n a_i y_i$. Then
\begin{equation}
    \begin{aligned}
        \varphi(f)(x) = (\Tilde{f} \circ \phi)(x) = \Tilde{f}\left(\sum_{i = 1}^n a_i y_i\right) = \sum_{i = 1}^n a_i f(y_i) =\\
        = \sum_{i = 1}^n a_i g(y_i) = \Tilde{g}\left(\sum_{i = 1}^n a_i y_i\right) = (\Tilde{g} \circ \phi)(x)
    \end{aligned}
\end{equation}
which proves the claim. The linearity of $\varphi$ is straightforward. Finally, $\varphi$ is bijective because
$$
C_p(X) \ni h \mapsto \psi(h) = (\Tilde{h} \circ \phi^{-1}) \upharpoonright Y \in C_p(Y)
$$
is its inverse.
\end{proof}

\subsection{Scattered-like properties and continuous linear surjections of $C_p$-spaces}

Recall that a space $X$ is \textit{scattered} if every nonempty subset of $X$ has an isolated point. A space $X$ is \emph{strongly $\sigma$-scattered} if $X$ is a countable union of closed scattered subspaces. Recently, several properties related to scatteredness have been actively studied, also in the context of linear surjections of function spaces $C_p(X)$ (see, e.g. \cite{KKL}, \cite{ELV}).
Given a family $\mathcal{A}$ of subsets of a space $X$, we say that $\mathcal{A}$ is \emph{point-finite} if for each $x\in X$ the collection $\{A\in \mathcal{A}:x\in A\}$ is finite. The family $\mathcal{A}$ is \emph{strongly point-finite} if there is a point-finite family $\{U_A:A\in \mathcal{A}\}$ of open subsets of $X$ satisfying $A\subseteq U_A$ for every $A\in \mathcal{A}$. Following \cite{Sa}, we say that a space $X$ has \emph{property $(\kappa)$} if every pairwise disjoint sequence of finite subsets of $X$ has an infinite strongly point-finite subsequence.
The following definition was given in \cite{KKL}:
\begin{defin}
A space $X$ satisfies property $\Delta_1$ (respectively, property $\Delta_2$) if every sequence $\{F_n: n\in \omega\}$ of pairwise disjoint finite (respectively, compact) subsets of $X$ is strongly point-finite.
\end{defin}
We have the following chain of implications between the properties defined above (see Theorems 2.15 and 3.9 in \cite{KKL}):
\begin{equation}\label{eq12}
\text{scattered}\Rightarrow \text{strongly $\sigma$-scattered}\Rightarrow \Delta_2 \Rightarrow \Delta_1 \Rightarrow \text{($\kappa$)}.
\end{equation}

It was established in \cite[Theorem 3.16]{KKL} that property $\Delta_1$ is preserved by continuous linear surjections between $C_p(X)$ spaces. It seems that the question of whether an analogous result holds for the other properties listed in~\eqref{eq12} has remained (and for some still remains) open. Below we shall prove the corresponding assertion for property $(\kappa)$ and strong $\sigma$-scatteredness. This answers a question of Leiderman (personal communication).

\begin{prop}\label{proposition_property_kappa}
Suppose that there is a continuous linear surjection of $C_p(X)$ onto $C_p(Y)$. If $X$ satisfies property $(\kappa)$, then so does $Y$.
\end{prop}

\begin{proof}
Let $\phi \colon C_p(X) \to C_p(Y)$ be a continuous linear surjection. Fix a sequence $\Ss = \{S_n:n=1,2,\ldots\}$ of nonempty finite pairwise disjoint subsets of $Y$. Put
$$F_0=\emptyset \quad \mbox{and}\quad F_n=\bigcup_{y \in S_n} \supp_{\phi}(y)\mbox{ for }n=1,2,\ldots.$$
Clearly, all sets $F_n$ are finite subsets of $X$.

Let $n_0=0$.
Inductively, construct an increasing sequence $n_1<n_2<\dotsb$ of natural numbers such that the following condition is satisfied for $k=1,2,\ldots$:
\begin{equation}\label{condition}
 \forall y\in S_{n_k}\;\;\supp_\phi(y)\not\subseteq F_0\cup\dotsb \cup F_{n_{k-1}}
\end{equation}
Take $n_1=1$ and suppose that for some $k\geq 1$ the numbers $n_1<\dotsb<n_k$ are already defined in such a way that condition \eqref{condition} is satisfied. In order to find $n_{k+1}$ let us consider the following finite subset $F$ of $X$:
$$F=F_0\cup\dotsb \cup F_{n_{k}}.$$
The sequence $\mathcal{S}$ consists of pairwise disjoint nonempty sets, so by Lemma \ref{lem:fin} there must be $n_{k+1}>n_k$ such that
if $y\in S_{n_{k+1}}$, then $\supp_\phi(y)\not\subseteq F$. This gives \eqref{condition} for $n_{k+1}$ and finishes the inductive construction.

To prove that $Y$ has property $(\kappa)$, it suffices to find a subsequence of the sequence
$\{S_{n_k}:k=1,2,\ldots\}$ which is strongly point-finite. To avoid multiple indexing let us put
$$T_k=S_{n_k},\mbox{ for }k=1,2,\ldots.$$
Henceforth, we will work with the
sequence $\{T_k:k=1,2,\ldots\}$ of pairwise disjoint finite subsets of $Y$ and we will prove that it admits a strongly point-finite subsequence.

For $k=1,2\ldots$, define
$$A_k=F_{n_k}\setminus \left( F_{n_0}\cup\dotsb \cup F_{n_{k-1}} \right).$$
The sequence $\{A_k:k=1,2,\ldots\}$ consists of pairwise disjoint finite subsets of $X$. So by the property $(\kappa)$ of $X$, there is a subsequence $\{A_{k_i}:i=1,2,\ldots\}$ that admits a point-finite open expansion $\{U_{i}:i=1,2,\ldots\}$, i.e., $A_{k_i}\subseteq U_i$. Let
$$
V_i=\supp_{\phi}^{-1}(U_i)=\{y \in Y \colon \supp_{\phi}(y) \cap U_i \neq \emptyset\}, \mbox{ for }i=1,2,\ldots.
$$

By lower semicontinuity of $\supp_\phi$, the set $V_i$ is open in $Y$ for every $i$.
It follows from \eqref{condition} that if $y\in T_{k_i}$, then
$$\supp_\phi(y)\cap A_{k_i}\neq\emptyset.$$
Hence, $T_{k_i}\subseteq V_i$. It remains to check that the family $\{V_i:i=1,2,\ldots\}$ is point-finite. To this end, pick $y\in Y$. If $y\in V_m$, then $\supp_\phi(y)$ meets $U_m$. Since the set $\supp_\phi(y)$ is finite and the family $\{U_i:i=1,2,\ldots\}$ is point-finite, the point $y$ can belong only to finitely many sets $V_i$.
\end{proof}

%
Moreover, it turns out that we may drop the assumption that $\phi$ is linear at the cost of assuming that $\phi$ is open.
Recall that a topological space $X$ satisfies \textit{property~$(B)$} if there exists a countable family $\{A_n \colon n \in \omega\}$ of closed nowhere dense subsets of $X$ such that for any compact subset $K \subseteq X$ there is $n \in \omega$ with $K\subseteq A_n$.
 
\begin{prop}
Suppose that $\phi \colon C_p(X) \to C_p(Y)$ is an open continuous surjection. If $X$ satisfies $(\kappa)$ then $Y$ satisfies $(\kappa)$ as well.
\end{prop}

\begin{proof}
Assume that $Y$ does not satisfy $(\kappa)$; we will show that $X$ does not either. By \cite[Theorem 1.2]{KKM}, the space $C_p(Y)$ satisfies property $(B)$. Let $\{A_n  \colon  n \in \NNN\}$ be a family of closed nowhere dense subsets of $C_p(Y)$ such that for any compact subset $K \subseteq C_p(Y)$, there exists $n \in \NNN$ for which $K \subseteq A_n$.

Since $\phi$ is open and continuous, the family $\{\phi^{-1}(A_n)  \colon  n \in \NNN\}$ consists of closed and nowhere dense subsets of $C_p(X)$. It is also straightforward to check that if $L$ is a compact subset of $C_p(X)$ then there is $n \in \NNN$ for which $L \subseteq \phi^{-1}(A_n)$. Finally, from Theorem 1.1 in \cite{KKM} we derive that $X$ does not satisfy $(\kappa)$ as required.
\end{proof}

%

Now, let us prove Theorem~\ref{thm:Cp-scattered}, which asserts that for any Tychonoff space, strong $\sigma$-scatteredness is invariant under continuous linear surjections between function spaces.

\begin{proof}[Proof of Theorem \ref{thm:Cp-scattered}]
Let $\phi \colon C_p(X) \to C_p(Y)$ be a continuous linear surjection. Write $X = \bigcup_{k \in \NNN} X_k$, where $X_k$ is scattered and closed in $X$ for all $k\in \NNN$. Since the union of two closed scattered sets is scattered, without loss of generality we may assume that $X_k \subseteq X_{k+1}$ for all $k \in \NNN$.\smallskip

For each $k \in \NNN$ define
\[
Y_k = \{y \in Y  \colon  \supp_{\phi}(y) \subseteq X_k\}
\]
and notice that $Y = \bigcup_{k \in \NNN} Y_k$. Since $X_k$ is closed, it follows from the lower semicontinuity of $\supp_{\phi}$ that each set $Y_k$ is closed in $Y$.

Now, for $n \in \NNN$ let
\[
Y_{k,n} = \{y \in Y_k  \colon  |\supp_{\phi}(y)| \leq n\}.
\]
Again, using the lower semicontinuity of $\supp_{\phi}$, it is easy to verify that the set $Y_{k,n}$ is closed in $Y_k$. Hence, all sets $Y_{k,n}$ are closed in $Y$. It is enough to show the following
\begin{claim}
The set $Y_{k,n}$ is scattered for all $k$ and $n$.
\end{claim}
\begin{proof}
Fix any nonempty subset $A \subseteq Y_{k,n}$. We will show that there is a nonempty open subset $V$ of $Y$ such that $V\cap A$ is finite. To this end, let
\[
m = \max\{|\supp_{\phi}(y)|  \colon  y \in A\}.
\]
The number $m$ is well defined since for all $y \in A$ we have $|\supp_{\phi}(y)| \leq n$. Pick any point $a_0 \in A$ such that $|\supp_{\phi}(a_0)| = m$ and denote
\[
\supp_{\phi}(a_0) = \{x_1, \dots, x_m\}.
\]
Fix pairwise disjoint open neighborhoods $U_1, \dots, U_m$ of the points $x_1, \dots, x_m$, respectively, and let
\[
V_0 = \bigcap_{i=1}^m \supp_{\phi}^{-1}(U_i).
\]
Note that if $y \in V_0 \cap A$ then the support of $y$ hits each $U_i$ exactly once.\smallskip

We will pick:
\begin{itemize}
\item nonempty open subsets $V_1, \dots, V_m$ of $Y$,
\item points $t_1 \in U_1, \dots, t_m \in U_m$, 
\item points $a_1 \in V_1 \cap A, \dots, a_m \in V_m \cap A$, 
\end{itemize}
in such a way that if $1 \leq j \leq i \leq m$ and $y \in V_i \cap A$, then
\[
\supp_{\phi}(y) \cap U_j = \{t_j\},
\]
while the points $a_i$ will witness that the sets $V_i \cap A$ are nonempty.\smallskip

Suppose that for some $0 \leq i < m$, the tuples of points $(t_1, \dots, t_i)$, $(a_0, \dots, a_i)$ and the sets $V_0, V_1, \dots, V_i$ are already defined (if $i = 0$, treat the tuple $(t_1, \dots, t_i)$ as empty) and consider the following set
\[
T = \bigcup\{ \supp_{\phi}(y) \cap U_{i + 1}  \colon  y \in V_i \cap A \}.
\]
Since $T$ is a subset of a scattered space $X_k$, it contains a relatively isolated point $t_{i + 1}$. Take $a_{i + 1} \in V_i \cap A$ satisfying $\supp_{\phi}(a_{i + 1}) \cap U_{i + 1} = \{t_{i + 1}\}$. Since $t_{i + 1}$ is relatively isolated in $T$, there exists an open set $W \subseteq X$ such that $W \cap T = \{t_{i + 1}\}$. Finally, define
\[
V_{i + 1} = \supp_{\phi}^{-1}(W\cap U_{i+1}) \cap V_i,
\]
and note that $a_{i + 1} \in V_{i + 1}$. It is readily seen that if $y \in V_{i + 1} \cap A$ then $\supp_{\phi}(y) \cap U_{j} = \{t_j\}$ for $j = 1, \dots, i + 1$. This finishes the induction.\smallskip

Let $V = V_m$. From the construction above, each $y \in V \cap A$ has the same support. Namely, $\supp_{\phi}(y) = \{t_1, \dots, t_m\}$ for all $y\in V\cap A$. Hence, by Lemma \ref{lem:fin} the set $V\cap A$ is finite.
\end{proof}

\end{proof}

%





\subsection{On the preservation of scatteredness under the $A$-equivalence relation}
In this section, we shall prove Theorem \ref{thm_scattered_main}. In fact, as we have already mentioned in the Introduction, we shall show the following slightly more general result.

\begin{thm}
Let $X$ and $Y$ be $A$-equivalent spaces and suppose that every closed subset of $X$ and $Y$ has a $W$-point. Then $X$ is scattered if and only if $Y$ is scattered.
\end{thm}

\begin{proof}
This proof is very similar to that of Theorem \ref{thm:Cp-scattered}. By symmetry, it is enough to prove that if $Y$ is scattered, then $X$ is scattered as well. To this end, let $\phi \colon A(X) \to A(Y)$ be a topological isomorphism. Fix any nonempty $A \subseteq X$. We aim at finding a nonempty open subset $V \subseteq X$ such that the set $A \cap V$ is finite. Without loss of generality we may assume that $A$ is closed (if that is not the case, then we take its closure).

From our assumption, the set $A$ has a $W$-point, say $x$. By Corollary \ref{cor:bdd_supp}, there is an open neighborhood $E \subseteq X$ of $x$ and a natural number $M$ satisfying
$$
\forall z \in E \quad |\supp_{\phi}(z)| \leq M.
$$
In particular, for all $z \in A \cap E$ the size of the support of $z$ does not exceed $M$. Now we need to repeat the proof of the Claim in Theorem \ref{thm:Cp-scattered} to obtain a nonempty relatively open set $V \subseteq A$ and a finite set $\{t_1, \dots, t_m\} \in [Y]^m$ with the property that
$$
\forall a \in V \quad \supp_{\phi}(a) = \{t_1, \dots, t_m\},
$$
where $m = \max\{|\supp_{\phi}(z)| \colon z \in A \cap E\}$. Finally, from Lemma \ref{lem:fin_grp} we derive the conclusion that the set $V$ must be finite, which finishes the proof of this theorem.
\end{proof}

Let us conclude the paper with the following remark:

\begin{rmk}\label{final_remark}
It follows from Theorem \ref{thm:Cp-scattered} that answering Problem~\ref{problem1} in the affirmative resolves Problem~\ref{problem2} in the affirmative as well. To see this, first note that any scattered space is hereditarily Baire. Indeed, let $X$ be scattered. Since the property of being scattered is hereditary, it suffices to show that $X$ is Baire. To this end, fix a sequence $\{A_n \colon n \in \omega\}$ of closed nowhere dense subsets of $X$. Striving for a contradiction, suppose that $\bigcup_{n \in \omega} A_n$ contains a nonempty open subset $U$.  Since $X$ is scattered, $U$ contains an isolated point $x \in U$, which is also isolated in $X$ because $U$ is open. Thus, $x \in U \subseteq \bigcup_{n \in \omega} A_n$. However, no set $A_n$ can contain an isolated point of $X$ because each $A_n$ is nowhere dense, a contradiction.

Next, note that if $Y$ is hereditarily Baire and strongly $\sigma$-scattered, then $Y$ is scattered. Indeed, write $Y=\bigcup_{n\in \omega} Y_n$, where each $Y_n$ is scattered and closed in $Y$. Aiming at a contradiction, suppose that $Y$ is not scattered. Then there is a closed crowded subset $C$ of $Y$, which must be Baire by our assumption. It follows that there exists $n\in \omega$ such that $Y_n\cap C$ contains a nonempty open subset $V$ of $C$. Since $Y_n$ is scattered, $V$ contains an isolated point, which is also isolated in $C$ because $V$ is open in $C$. This contradicts the fact that $C$ is crowded.
\end{rmk}


\bibliographystyle{siam}
\bibliography{bib.bib}

\end{document}